\documentclass[10pt]{amsart}
\usepackage{amsmath, amssymb}
 \usepackage{mathrsfs}
\newcommand{\no}[1]{#1}
\renewcommand{\no}[1]{}
\no{\usepackage{times}\usepackage[subscriptcorrection, slantedGreek, nofontinfo]{mtpro}
\renewcommand{\Delta}{\upDelta}}
\usepackage{color}

\date{\today}
\newtheorem{theorem}{Theorem}[section]

\newtheorem{lemma}{Lemma}[section]

\newtheorem{corollary}{Corollary}[section]

\theoremstyle{remark}
\newtheorem{remark}{Remark}[section]

\numberwithin{equation}{section}

\title[Gaussian type bounds]{Gaussian lower bound for the Neumann Green function of a general parabolic operator}

\author[Mourad Choulli]{Mourad Choulli}
\address{Institut \'Elie Cartan de Lorraine, UMR CNRS 7502, Universit\'e de Lorraine-Metz, F-57045 Metz cedex 1, France}
\email{mourad.choulli@univ-lorraine.fr, laurent.kayser@univ-lorraine.fr}

\author[Laurent Kayser]{Laurent Kayser}

\date{}

\begin{document}

\begin{abstract}
Based on the fact that the Neumann Green function can be constructed as a perturbation of the fundamental solution by a single-layer potential, we establish a Gaussian lower bound for the Neumann Green function for a general parabolic operator. We build our analysis  on classical tools coming from the construction of a fundamental solution of a general parabolic operator by means of the so-called parametrix method.  At the same time we provide a simple proof for Gaussian two-sided bounds for the fundamental solution.


\end{abstract}

\maketitle

\tableofcontents


\section{Introduction}

Let $\Omega$ be a bounded domain of $\mathbb{R}^n$ with $C^{1,1}$-smooth boundary. Let $t_0<t_1$, we set $Q=\Omega \times (t_0,t_1)$ and we consider the second order differential operator
\[
L=a_{ij}(x,t)\partial _{ij}^2+b_k(x,t)\partial _k+c(x,t)-\partial _t.
\]
Here and henceforth we use the usual Einstein summation convention for repeated indices.

\smallskip
We make the following assumptions on the coefficients of $L$: 
\begin{align*}
&(i)\; \mbox{the matrix}\; (a_{ij}(x,t))\; \mbox{is symmetric for any}\; (x,t)\in \overline{Q}, \hskip 9cm
\\
&(ii)\; a_{ij}\in W^{1,\infty}(Q),\;\; b_k,\; c \in C([t_0,t_1], C^1 (\overline{\Omega})),
\\
&(iii)\; a_{ij}(x)\xi _i\xi _j\geq \lambda |\xi |^2,\;\; (x,t)\in \overline{Q}  ,\; \xi \in \mathbb{R}^n,
\\
&(iv)\; \|a_{ij}\|_{W^{1,\infty}(Q)}+\|b_k\|_{L^\infty (Q)}+\|c\|_{L^\infty (Q)}\leq A,
\end{align*}
where $\lambda >0$ and $A>0$ are two given constants.

\smallskip
These assumptions are surely not the best possible if one wants to construct a fundamental solution or a Green function. But they are sufficient to carry out our analysis.

\smallskip
Since we will use the fundamental solution in the whole space, we begin by extending the coefficients of $L$ in a neighborhood $\widetilde{\Omega}$  of $\overline{\Omega}$ to coefficients having the same regularity. We observe that this is possible in view of the regularity of $\Omega$. For sake of simplicity, we keep the same symbols for the extended coefficients. We may also assume that the ellipticity condition holds for the extended coefficients with the same constant $\lambda$. Pick $\psi \in C_0^\infty (\widetilde{\Omega})$ satisfying $0\leq \psi \leq 1$ and $\psi =1$ in a neighborhood of $\overline{\Omega}$. We set
\[
\widetilde{a}_{ij}=a_{ij}\psi +\lambda \delta _{ij}(1-\psi ),\;\; \widetilde{b}_k=b_k\psi ,\;\; \widetilde{c}=c\psi 
\]
and
\[
\widetilde{L}=\widetilde{a}_{ij}(x,t)\partial _{ij}^2+\widetilde{b}_k(x,t)\partial _k+\widetilde{c}(x,t)-\partial _t.
\]
Clearly, the coefficients of $\widetilde{L}$ satisfy the same assumptions as those of $L$. So in the sequel we will use  the same symbol $L$ for $L$ or its extension $\widetilde{L}$.

\medskip
We recall that the function
\[
\mathscr{G}(x,t)=(4\pi t)^{-n/2}e^{-\frac{|x|^2}{4t}},\;\; x\in \mathbb{R}^n,\; t>0,
\]
is usually called the Gaussian Kernel. We set 
\[
\mathscr{G}_c(x,t)=c^{-1}\mathscr{G}(\sqrt{c}x,t),\;\; c>0.
\]
It is important to observe that the map $c\rightarrow \mathscr{G}_c$ is non increasing.

\medskip
We are interested in  establishing a Gaussian lower bound for the Neumann Green function associated to the operator $L$. More specifically, denoting by $G$ the Neumann Green function  for $L$, we want to prove an estimate of the form 
 \[
 \mathscr{G}_C(x-\xi ,t-\tau )\leq G(x,t;\xi ,\tau ),\;\; (x,t;\xi ,\tau )\in Q^2,\, t>\tau ,
 \]
where the constant $C$  depends only on $\Omega$, $\lambda$, $T=t_1-t_0$ and $A$.

\smallskip
We succeed in proving that the above  Gaussian lower bound holds true provided that $\Omega$ satisfies the chain condition. That is, there exists a constant $c>0$ such that for any two points $x$, $y\in \Omega$ and for any positive integer $m$ there exists a sequence $(x_i)_{0\leq i\leq m}$ of points in $\Omega$ such that $x_0=x$, $x_m=y$ and 
\[
|x_{i+1}-x_i|\leq \frac{c}{m}|x-y|,\;\; i=0,\ldots ,m-1.
\]
The sequence $(x_i)_{0\leq i\leq m}$ is referred to as a chain connecting $x$ and $y$.

\smallskip
We see that any convex subset of $\mathbb{R}^n$ satisfies the chain condition with $c=1$. In two dimensional case, the spherical shell $\mathcal{C}=B(0,2)\setminus \overline{B(0,1)}$ has the chain property with $c=\sqrt{2}$. This follows from the fact that any two points of $\mathcal{C}$ can be connected by a broken line consisting of two segments parallel to axes of coordinates. 

\smallskip
We point out that a $C^{1,1}$-smooth domain does not possess necessarily the chain condition.

\smallskip
To our knowledge a Gaussian lower bound has never been established before for the Neumann Green function of a general parabolic operator. Moreover, even in the case of parabolic operators with time-independent coefficients, we can quote only three references: \cite{Cha} when the domain is convex, \cite{COY} for smooth domains and \cite{LY} for a compact Riemannian manifold with boundary whose Ricci curvature is bounded from above and its boundary is convex.

\smallskip
A Gaussian upper bound for a general parabolic operator in divergence form was proved by Daners \cite{Dan}.  In  \cite{CK}, Choi and Kim obtained a Gaussian upper bound for a system of operators in divergence form under the assumption that the corresponding Neumann boundary value problem possesses a De Giorgi-Nash-Moser type estimate at the boundary. In \cite{BCS}, the authors established a gaussian upper bound for a Neumann Green function corresponding to a time-dependent domain.

\medskip
The problem is quite different for a Dirichlet Green function since the latter vanishes on the boundary. One can prove in an obvious manner, with the help of the parabolic maximum principle, that a Dirichlet Green function is non negative and dominated pointwise  by a fundamental solution and so it has a Gaussian upper bound. Aronson \cite[Theorem 8, page 670]{Ar2} established an interior Gaussian lower bound for a Dirichlet Green function. It is worthwhile to mention that \cite[Theorem 8, page 670]{Ar2}  can be used to extend the results  of \cite[Section 3]{FS} to a general parabolic operator. In other words, one can obtain a proof of a continuity theorem by Nash \cite{Na} and Moser-Harnack inequality \cite{Mo} for a general divergence form parabolic operator, since they rely on two-sided Gaussian bounds for the fundamental solution. Later, Cho \cite{Cho}, Cho, Kim and Park \cite{CKP} extended this result to a global weighted Gaussian lower bound involving the distance to the boundary. A Gaussian lower bound for a Dirichlet Green function when the Euclidian distance is changed by a geodesic distance was proved by van den Berg \cite{Va1, Va2}.

\smallskip
For parabolic operators with time-independent coefficients,  a fundamental solution or a Green function is reduced to a heat kernel. We mention that there is a tremendous literature dealing with Gaussian bounds for heat kernels. We quote the  classical books by Davies \cite{Dav}, Grigor'yan \cite{Gr}, Ouhabaz \cite{Ou} Saloff-Coste \cite{Sa} and Stroock \cite{St}, but of course there are many other references on the subject.

\smallskip
As we  said in the summary, the main ingredient in our analysis relies on the classical construction of the fundamental solution by means of the  so-called parametrix method. We revisit this construction in the next section and we derive from it Gaussian two-sided bounds for the fundamental solution. In Section 3, we prove a Gaussian lower bound for the Neumann Green function. To do so, we construct the Neumann Green function as a perturbation of the fundamental solution by a single-layer potential. The  Gaussian lower bound is then derived from the smoothing effect of the single-layer potential. 

\section{The parametrix method revisited}

We are concerned in this section with Gaussian two-sided bounds for the fundamental solution of $Lu=0$. For a systematic study of fundamental solutions, we refer to the classical monographs by A. Friedman \cite{Fr} and O. A. Ladyzhenskaja, V. A. Solonnikov and N. N. Ural'tzeva \cite{LSU}.

\smallskip
In the sequel $P=\mathbb{R}^n\times (t_0,t_1)$.

\smallskip
We recall that a fundamental solution of $Lu=0$ in $P$ is a function $E (x,t;\xi ,\tau )$ which is $C^{2,1}$ in $P^2\cap\{t>\tau\}$, which satisfies 
\[
LE (\cdot\, ,\cdot ;\xi ,\tau )=0\; \mbox{in}\; \mathbb{R}^n \times \{\tau <t\leq t_1\},\; \mbox{for any}\; (\xi ,\tau )\in \mathbb{R}^n \times [t_0,t_1[
\] 
and, for any $f\in C_0(\mathbb{R}^n)$,
\[
\lim_{t\searrow \tau}\int_{\mathbb{R}^n} E (x,t;\xi ,\tau )f(\xi )d\xi =f(x),\;\; x\in \mathbb{R}^n.
\]
In this definition, we can also take a larger class of functions $f$. Namely, a class of continuous functions satisfying a certain growth condition at infinity  (see for instance \cite[formulas $(6.1)$ and $(6.2)$, page 22]{Fr}).

\medskip
The construction of a fundamental solution by means of the so-called parametrix method was initiated by E. E. Levi \cite{Le}. Let $a=(a^{ij})$ be the inverse matrix of $(a_{ij})$, $|a|$ the determinant of $a$ and 
\[
Z(x,t;\xi ,\tau )=[4\pi (t-\tau )]^{-n/2}\sqrt{|a(\xi ,\tau )|}e^{-\frac{a(\xi ,\tau )(x-\xi )\cdot (x-\xi )}{4(t-\tau )}},\;\; (x,t;\xi ,\tau)\in P^2\cap\{t>\tau\}.
\]
This function is called the  parametrix. It satisfies
\begin{equation}\label{2.1}
L_0Z(\cdot \, ,\cdot\, ,\xi ,\tau )=0\; \mbox{in}\;  \mathbb{R}^n\times \{\tau <t\leq t_1\}\; \mbox{for any}\; (\xi ,\tau )\in \mathbb{R}^n \times [t_0,t_1[,
\end{equation}
where
\[
L_0=a_{ij}(\xi ,\tau )\partial ^2_{ij} -\partial _t.
\]
When $(\xi ,\tau )$ are fixed, $L_0$ is considered as a constant coefficients operator with respect to $(x,t)$.

\smallskip
In the parametrix method we seek $E$, a fundamental solution of $Lu=0$ in $P$, of the form
\begin{equation}\label{2.2}
E (x,t;\xi ,\tau )=Z (x,t;\xi ,\tau )+\int_\tau^t\int_{\mathbb{R}^n} Z (x,t;\eta ,\sigma )\Phi (\eta ,\sigma ;\xi ,\tau )d\eta d\sigma ,
\end{equation}
where $\Phi$ is to be determined in order to satisfy $LE (\cdot\, ,\cdot\,  ;\xi ,\tau )=0\; \mbox{for any}\; (\xi ,\tau )\in \mathbb{R}^n \times [t_0,t_1[$.

\smallskip
Following  \cite[Formulas $(4.4)$ and $(4.5)$, page 14]{Fr}, $\Phi$ is given by the series
\[
\Phi=\sum_{\ell=1}^\infty \Phi_\ell,
\] 
where $\Phi _1(x,t;\xi ,\tau )=LZ(x,t;\xi ,\tau )$ and
\[
\Phi_{\ell +1}(x,t;\xi ,\tau )=\int_\tau^t\int_{\mathbb{R}^n} \Phi_1(x,t ;\eta ,\sigma )\Phi_\ell (\eta ,\sigma ;\xi ,\tau )d\eta d\sigma ,\;\; \ell \geq 1.
\]
Here,  for simplicity, we write $LZ(x,t;\xi ,\tau )$  instead of $[LZ(\cdot \, , \cdot \, ; \xi ,\tau )](x,t)$.

\smallskip
Let $d_i$, $1\leq i\leq n$, given by
\[
 d_i=d_i(x,t;\xi ,\tau )=-\frac{a^{ij}(\xi ,\tau )(x_j-\xi _j)}{2(t-\tau)},\;\;  (x,t;\xi ,\tau)\in P^2\cap\{t>\tau\}.
 \]
 Then
 \[
 \partial _i Z=d_iZ,\;\; \partial_{ij}^2Z=\left[ -\frac{a^{ij}(\xi ,\tau )}{2(t-\tau)}+d_jd_i\right]Z.
 \]
 Therefore, taking into account \eqref{2.1}, we get
 \[
 LZ=LZ-L_0Z=\left\{ \left( a_{ij}(x,t)- a_{ij}(\xi ,\tau )\right)\left[ -\frac{a^{ij}(\xi ,\tau )}{2(t-\tau)}+d_jd_i\right]+b_kd_k +c\right\}Z.
 \]
 We write $LZ=\Psi Z$, where
 \[
 \Psi =\left( a_{ij}(x,t)- a_{ij}(\xi ,\tau ) \right) \left[ -\frac{a^{ij}(\xi ,\tau )}{2(t-\tau)}+d_jd_i \right]+b_kd_k +c.
 \]
 Let
 \[
 M=\max_{i,j}\| a_{ij}\|_{W^{1,\infty}(Q)},\;\; N=\max(\max_k\|b_k\|_{L^\infty (Q)},\|c\|_{L^\infty (Q)},1).
 \]
 
 Since
 \begin{align*}
 &|d_i|\leq \frac{|x-\xi|}{2\lambda (t-\tau )},
 \\
 &|a_{ij}(x,t)- a_{ij}(\xi ,\tau )|\leq M(|x-\xi |+t-\tau ),
 \end{align*}
 we have
 \begin{equation}\label{2.3}
 |\Psi (x,t;\xi ,\tau )|\leq N\frac{1}{\sqrt{t-\tau}}\mathscr{P}\left( \frac{|x-\xi |}{\sqrt{t-\tau}}\right).
 \end{equation}
 Here $\mathscr{P}$ is a polynomial function of degree less than three whose coefficients depend only on $M$. 

\smallskip
Unless otherwise stated, all the constants we  use now do not depend on $N$. 

\smallskip
In light of \eqref{2.3} we obtain
 \[
 |LZ|\leq CN(t-\tau)^{-(n+1)/2}P(\rho )e^{-(\lambda /4)\rho^2}=CN(t-\tau)^{-(n+1)/2}\left[P(\eta )e^{-(\lambda /8)\rho ^2}\right]e^{-(\lambda /8)\rho^2},
 \]
with
 \[
 \rho =\frac{|x-\xi |}{\sqrt{t-\tau}}.
 \]
 But the function $\rho \in (0,+\infty ) \longrightarrow P(\rho )e^{-(\lambda /8)\rho ^2}$ is bounded. Consequently, 
 \begin{equation}\label{2.4}
 |\Phi _1(x,t;\xi ,\tau )|=|LZ(x,t;\xi ,\tau )|\leq N\widetilde{C}(t-\tau)^{-(n+1)/2}e^{-\frac{\lambda ^\ast |x-\xi |^2}{t-\tau }},
 \end{equation}
 where $\lambda ^\ast=\lambda /8$.
 
 \smallskip
 The following lemma will be useful in the sequel. Its proof  is given in \cite[page 15]{Fr}.
 
 \begin{lemma}\label{lemma2.1}
 Let $c>0$ and $-\infty <\gamma,\beta <n/2 +1$. Then
 \begin{align*}
 \int_\tau ^t\int_{\mathbb{R}^n}(t-\sigma )^{-\gamma}e^{-\frac{c|x-\eta |^2}{t-\sigma }}&(\sigma -\tau )^{-\beta}e^{-\frac{c|\eta -\xi |^2}{\sigma -\tau}}d\eta d\sigma 
\\
&=\left(\frac{4\pi}{c}\right)^{n/2}B(n/2-\gamma +1,n/2-\beta +1)(t-\tau)^{n/2+1-\gamma-\beta}e^{-\frac{c|x-\xi |^2}{t-\tau }},
 \end{align*}
 where $B$ is the usual beta function.
 \end{lemma}
 
 We want to show
 \begin{equation}\label{2.5}
 |\Phi _\ell (x,t;\xi ,\tau )|\leq (N\widetilde{C})^\ell \widehat{C}^{\ell -1}(t-\tau )^{-(n+2-\ell )/2}\prod_{j=1}^{\ell -1}B(1/2,j/2)e^{-\frac{\lambda ^\ast |x-\xi |^2}{t-\tau}},\;\; \ell \geq 2.
 \end{equation}
 Here $\widetilde{C}$ is the same constant as in \eqref{2.4} and $\widehat{C}=\left(\frac{4\pi}{\lambda ^\ast}\right)^{n/2}$.
 
\smallskip
As
 \[
 \Phi _2(x,t;\xi ,\tau )=\int_\tau ^t\int_{\mathbb{R}^n}\Phi_1(x,t;\eta ,\sigma )\Phi _1(\eta ,\sigma ;\xi ,\tau )d\eta d\sigma ,
 \]
 estimate \eqref{2.4} and Lemma \ref{lemma2.1} with $\gamma =\beta =n/2+1$ show that \eqref{2.5} holds true with $\ell =2$. The general case follows by an induction argument in $\ell$. Indeed, using
 \[
 \Phi _{\ell +1}(x,t;\xi ,\tau )=\int_\tau ^t\int_{\mathbb{R}^n}\Phi_1(x,t;\eta ,\sigma )\Phi _\ell (\eta ,\sigma ;\xi ,\tau )d\eta d\sigma ,
 \]
 \eqref{2.4}, \eqref{2.5} for $\ell$ and Lemma \ref{lemma2.1} with $\gamma =n/2+1$ and $\beta =(n+2-\ell )/2$, we obtain easily that \eqref{2.5} holds true with $\ell +1$ in place of $\ell$.
 
 \smallskip
If $\Gamma$ is the usual gamma function, we recall that
 \[
B(1/2,j/2)=\frac{\Gamma (1/2)\Gamma (j/2)}{\Gamma ((j+1)/2)}.
 \]
 Therefore
 \begin{equation}\label{2.6}
 \prod_{j=1}^{\ell -1}B(1/2,j/2)=\frac{\Gamma (1/2)^\ell}{\Gamma (\ell /2)}=\frac{\sqrt{\pi}^\ell}{\Gamma (\ell /2)}.
 \end{equation}
Hence, \eqref{2.4}-\eqref{2.6} entail 
 \begin{equation}\label{2.7}
 |\Phi (x,t;\xi ,\tau )|\leq \sum_{\ell \geq 1}|\Phi _\ell (x,t;\xi ,\tau )|\leq N\widetilde{C}(1+S)(t-\tau)^{-(n+1)/2}e^{-\frac{\lambda ^\ast |x-\xi |^2}{t-\tau }},
 \end{equation}
 with
 \[
 S=\sum_{\ell \geq 1}\left[CN (t-\tau )^{1/2}\right]^\ell /\Gamma ((\ell +1)/2).
 \]
 We have  $\Gamma ((\ell +1)/2)=\Gamma (m+1/2)\geq \Gamma (m)=(m-1)!$ if $\ell =2m$ and $\Gamma ((\ell +1)/2)=\Gamma (m+1)=m!$ if $\ell =2m+1$. Then
 \begin{align*}
 S&=\frac{1}{\Gamma(3/2)}[CN(t-\tau)^{1/2}]^2+\sum_{m \geq 2}\frac{1}{\Gamma (m+1/2)}\left[CN (t-\tau )^{1/2}\right]^{2m} +\sum_{m \geq 0}\frac{1}{\Gamma (m+1)}\left[CN (t-\tau )^{1/2}\right]^{2m+1} 
 \\
 &\leq \frac{1}{\Gamma(3/2)}[CN(t-\tau)^{1/2}]^2+\sum_{m \geq 2}\frac{1}{(m-1)!}\left[CN (t-\tau )^{1/2}\right]^{2m} +\sum_{m \geq 0}\frac{1}{m!}\left[CN (t-\tau )^{1/2}\right]^{2m+1} .
 \end{align*}
 Whence
 \[
 1+S\leq \widetilde{C}e^{\widetilde{C}N^2(t-\tau )}.
 \]
Plugging this estimate into \eqref{2.7}, we obtain
  \begin{equation}\label{2.8}
 |\Phi (x,t;\xi ,\tau )|\leq  \widetilde{C}N(t-\tau)^{-(n+1)/2}e^{-\frac{\lambda ^\ast |x-\xi |^2}{t-\tau }+\widetilde{C}N^2(t-\tau )}.
 \end{equation}
 
With the help of Lemma \ref{lemma2.1}, estimate \eqref{2.8} yields
 \begin{equation}\label{2.9}
 \left| \int_\tau^t\int_{\mathbb{R}^n} Z (x,t;\eta ,\sigma )\Phi (\eta ,\sigma ;\xi ,\tau )d\eta d\sigma\right|\leq \widetilde{C}N(t-\tau)^{-(n-1)/2}e^{-\frac{\lambda ^\ast |x-\xi |^2}{t-\tau }+\widetilde{C}N^2(t-\tau )}.
 \end{equation} 
 
 Noting that this inequality can be rewritten as
 \[
  \left| \int_\tau^t\int_{\mathbb{R}^n} Z (x,t;\eta ,\sigma )\Phi (\eta ,\sigma ;\xi ,\tau )d\eta d\sigma\right|\leq \widetilde{C}\left[N(t-\tau)^{1/2}e^{-\widetilde{C}N^2(t-\tau )}\right](t-\tau)^{-n/2}e^{-\frac{\lambda ^\ast |x-\xi |^2}{t-\tau }+2\widetilde{C}N^2(t-\tau )}
 \]
and, using that $\rho \rightarrow \rho e^{-\widetilde{C}\rho ^2}$ is a bounded function on $[0,+\infty )$, we obtain
 \begin{equation}\label{2.9.1}
  \left| \int_\tau^t\int_{\mathbb{R}^n} Z (x,t;\eta ,\sigma )\Phi (\eta ,\sigma ;\xi ,\tau )d\eta d\sigma\right| \leq \widehat{C}(t-\tau)^{-n/2}e^{-\frac{\lambda ^\ast |x-\xi |^2}{t-\tau }+2\widetilde{C}N^2(t-\tau )}.
 \end{equation}
 
An immediate consequence of \eqref{2.9.1} is
 \begin{equation}\label{2.10}
 |E(x,t;\xi ,\tau )|\leq \widetilde{C}(t-\tau)^{-n/2}e^{-\frac{\lambda ^\ast |x-\xi |^2}{t-\tau }+\widetilde{C}N^2(t-\tau )}.
 \end{equation}
 
 \smallskip
 In the rest of this section, we forsake the explicit dependence on $N$. So the constants below may depend on $\Omega$, $\lambda$, $A$,  and $T$.
 
 \smallskip
 From \eqref{2.9}  we deduce in a straightforward manner that there exists $\delta >0$ such that
 \begin{equation}\label{2.11}
 E (x,t;\xi ,\tau )\geq \widehat{C}(t-\tau )^{-n/2},\;\; (x,t;\xi ,\tau )\in P^2,\; t>\tau ,\; \widetilde{C}|x-\xi |^2<t-\tau \leq \delta.
 \end{equation}
By  \cite[Theorem 11, page 44]{Fr}, $E$ is positive. Moreover, $E$ satisfies the following identity, usually called the reproducing property,
\begin{equation}\label{2.12}
E (x,t; \xi ,\tau )=\int_{\mathbb{R}^n} E (x,t ;\eta ,\sigma )E (\eta ,\sigma ;\xi ,\tau )d\eta ,\; \; x,\xi \in \mathbb{R}^n,\;\; t_0\leq \tau <\sigma <t\leq t_1.
\end{equation}
We can then paraphrase the proof of \cite[Theorem 2.7, page 334]{FS} to get a Gaussian lower bound for $E$ when $0<t-\tau \leq \delta$. To pass from $t-\tau \leq \delta$ to $t-\tau \leq T$, we use again an argument based on the reproducing property. We detail the same argument in the proof of Theorem \ref{theorem3.1}.

\smallskip
We sum up our analysis in the following theorem.

 \begin{theorem}\label{theorem2.1}
 The fundamental solution $E$ satisfies the Gaussian two-sided bounds:
 \begin{equation}\label{2.13}
  \mathscr{G}_C(x-\xi , t-\tau )\leq E (x,t;\xi ,\tau )\leq  \mathscr{G}_{\widetilde{C}}(x-\xi , t-\tau ),\; (x,t;\xi ,\tau )\in P^2\cap\{t>\tau\}.
 \end{equation}
 \end{theorem}
 
  \begin{remark}\label{remark2.1}
Let us assume that conditions $(i)$-$(iv)$ above hold in all of the whole space $\mathbb{R}^n\times \mathbb{R}$ instead of $Q$ only. Taking into account the exponential term in $N^2$ in \eqref{2.10}, we prove, once again with the help of the reproducing property, the following global estimate in time:
\begin{equation}\label{2.14}
e^{-\kappa N^2(t-\tau )}\mathscr{G}_C(x-\xi , t-\tau )\leq E (x,t;\xi ,\tau )\leq  e^{\kappa N^2(t-\tau )}\mathscr{G}_{\widetilde{C}}(x-\xi , t-\tau ),
\end{equation}
for some constant $\kappa >0$, where $(x,t;\xi ,\tau )\in (\mathbb{R}^n\times \mathbb{R})^2$.

\medskip
We point out that \eqref{2.14} does not give the two-sided Gaussian bounds by Fabes and Stroock \cite{FS} for the divergence form operator $\partial _i(a_{ij}(x,t)\partial _j\, \cdot \, )-\partial _t$ with ($C^\infty$-) smooth coefficients. This is not surprising since the arguments we used for proving  \eqref{2.14} are not well adapted to divergence form operator. We note however that the approach developed in \cite{FS} for establishing  Gaussian two-sided bounds is more involved.

\medskip
Gaussian two-sided bounds were obtained by S. D. Eidel'man and F. O. Porper \cite{EP} when the coefficients of $L$ satisfy the uniform Dini condition with respect to $x$. The main tool in \cite{EP} is a parabolic Harnack inequality. We refer also to \cite{Ar1}, \cite{Fa}, \cite{It} and \cite{NS}, where the reader can find various results on bounds for the fundamental solution.
 \end{remark}

 We mentioned in the Introduction that the Moser-Harnack inequality in \cite{FS} can be extended to a general divergence form parabolic operator. Let us show briefly how this Moser-Harnack inequality still holds for a general parabolic operator. First, we recall that a Dirichlet Green function was constructed in \cite[formula (16.7), page 408]{LSU} as a perturbation of the fundamental solution by a double-layer potential. Therefore, in light of  \cite[formula (16.10), page 409]{LSU} and \cite[estimate (16.14), page 411]{LSU}, we can assert that \cite[Lemma 5.1]{FS} remains true for our $L$. Next, paraphrasing the proofs of \cite[Lemma 5.2 and Theorem 5.4]{FS} (more detailed proofs are given in \cite{St}), we can state the following Moser-Harnack inequality.
 
 \begin{theorem}\label{theorem2.2}
 Let $\eta ,\mu ,\varrho \in (0,1)$. Then there is $M>0$, depending on $n$, $\lambda$, $A$, $\eta$, $\mu$  and $\varrho$ such that for all $(x,s)\in \mathbb{R}^n\times \mathbb{R}$, all $R>0$ and all non negative $u\in C^{2,1}(\overline{B}(x, R)\times [s- R^2,s])$ satisfying $Lu=0$ one has
 \[
 u(y,t)\leq Mu(x,s)\;\; \mbox{for all}\; (y,t)\in \overline{B}(x,\varrho R)\times [s-\eta R^2,s-\mu R^2].
 \]
 \end{theorem}
 

\section{Gaussian lower bound for the Neumann Green function}

We recall that the derivative of $U=U(x,t)$ at $(x,t)\in \partial \Omega \times [t_0,t_1]$ in the conormal direction is given by
\[
\partial _\nu U(x,t)= a_{ij}(x,t)\mathbf{n}_j(x)\partial _iU(x,t),
\]
where $\mathbf{n}(x)=(\mathbf{n}_1(x), \ldots ,\mathbf{n}_n(x))$ is the unit outward normal vector at $x$. 

\medskip
For $\tau \in [t_0,t_1[$, we set $Q_\tau =\Omega \times (\tau , t_1)$, $\Sigma _\tau=\partial \Omega \times (\tau ,t_1)$ and we consider the Neumann initial-boundary value problem (abbreviated to IBVP in the sequel) for the operator $L$:
\begin{equation}\label{3.1}
\left\{
\begin{array}{lll}
Lu =0\quad &\mbox{in}\; Q_\tau  ,
\\
u(\cdot\, ,\tau )=\psi &\mbox{in}\; \Omega ,
\\
\partial _\nu u=0 &\mbox{on}\; \Sigma_\tau.
\end{array}
\right.
\end{equation}

\smallskip
From \cite[Theorem 2, page 144]{Fr} and its proof, for any $\psi \in C_0^\infty (\Omega )$, the IBVP \eqref{3.1} has a unique solution $u\in C^{1,0}(\overline{Q}_\tau)\cap C^{2,1}(Q_\tau)$ given by
\begin{equation}\label{3.2}
u(x,t)=\int_\tau ^t\int_{\partial \Omega}E (x,t;\xi ,\sigma )\varphi (\xi ,\sigma )d \xi d\sigma +\int_\Omega E (x,t;\xi ,\tau )\psi (\xi )d\xi .
\end{equation}
Here
\begin{equation}\label{3.3}
\varphi (x,t)=F_\tau (x,t)-2\sum_{\ell \geq 1}\int_\tau ^t\int_{\partial \Omega} M_\ell (x,t;\xi ,\sigma )F_\tau (\xi ,\sigma )d\xi d\sigma  ,
\end{equation}
with
\begin{align*}
&F_\tau (x,t)=-2\int_\Omega \partial _\nu  E (x,t;\xi ,\tau ) \psi (\xi )d\xi ,
\\ 
&M_1=-2\partial _\nu E ,
\\
&M_{\ell +1}(x,t;\xi ,\tau )=\int_\tau ^t\int_{\partial \Omega}M_1(x,t;\eta ,\sigma )M_\ell (\eta ,\sigma ;\xi ,\tau )d\eta d\sigma .
\end{align*}

For $(x,t)\in \Sigma _\tau$ and $\xi \in \Omega$, let
\[
\mathcal{N}(x,t; \xi ,\tau )=-2\partial _\nu E (x,t;\xi ,\tau )-2\sum_{\ell \geq 1}\int_\tau ^t\int_{\partial \Omega}M_\ell (x,t;\eta ,\sigma )\partial _\nu E (\eta ,\sigma ;\xi ,\tau )d\eta d\sigma  .
\]
Assume for the moment (see the proof below) that
\begin{equation}\label{3.4}
\varphi (x,t)=\int_\Omega \mathcal{N}(x,t; \xi ,\tau )\psi (\xi )d\xi.
\end{equation}
We set
\begin{equation}\label{3.5}
G(x,t,\xi ,\tau )= \int_\tau ^t\int_{\partial \Omega}E (x,t;\eta ,\sigma )\mathcal{N}(\eta ,\sigma ;\xi ,\tau )d \eta d\sigma +E (x,t;\xi ,\tau ).
\end{equation}
It follows from Fubini's theorem that
\begin{equation}\label{3.6}
u (x,t)=\int_\Omega G(x,t;\xi ,\tau )\psi  (\xi )d\xi .
\end{equation}
The function $G$ is called the Neumann Green function associated to the equation $Lu=0$ in $Q$.

\medskip
We have, for any $0\leq \psi \in C_0^\infty (\Omega )$, $u \geq 0$, according to the maximum principle (see for instance \cite[Theorem 2.9, page 15]{Li} and remarks following it). Whence, $G\geq 0$. 

\smallskip
From the uniqueness of the solution of the IBVP \eqref{3.1}, we have also
\[
\int_\Omega G(x,t;\xi ,\tau )\psi  (\xi )d\xi =\int_\Omega G(x,t;\eta ,\sigma )d\eta \int_\Omega G(\eta ,\sigma ;\xi ,\tau )\psi  (\xi )d\xi \;\; \mbox{for any}\; \psi \in C_0^\infty (\Omega ),\; \tau <\sigma <t.
\]
Therefore,
\begin{equation}\label{3.7}
 G(x,t;\xi ,\tau )=\int_\Omega G(x,t;\eta ,\sigma )G(\eta ,\sigma ;\xi  ,\tau )d\eta ,\; \tau <\sigma <t.
\end{equation}
That is, $G$ has the reproducing property. 

\smallskip
We note that when $c=0$, $G$ satisfies in addition
\[
\int_\Omega G(x,t;\xi ,\tau )d\xi =1.
\]

\smallskip
The key point in the proof of our Gaussian lower bound for $G$ is the following lemma.

\begin{lemma}\label{lemma3.1}
For $1/2<\mu < \frac{n}{2}$, we have
\begin{equation}\label{3.8}
|\mathcal{N}(x,t;\xi ,\tau )|\leq C(t-\tau )^{-\mu}|x-\xi |^{-n+2\mu },\; (x,t)\in \Sigma _\tau,\; \xi \in \Omega ,\; x\neq \xi .
\end{equation}
\end{lemma}

The lemma below appears in  \cite[page 137]{Fr}  as Lemma 1. It is needed for proving Lemma \ref{lemma3.1}.

\begin{lemma}\label{lemma3.2}
Let $0< a,b <n-1$ with $a+b\neq n-1$. Then
\begin{equation}\label{3.9}
\int_{\partial \Omega}|x-\eta|^{-a}|\eta -\xi |^{-b}d\eta \leq \left\{ \begin{array}{ll} \widehat{C}|x-\xi |^{n-1-(a+b)}\;\; &\textrm{if}\;\; a+b>n-1\\ \widehat{C}\;\; &\textrm{if}\;\; a+b<n-1.\end{array}\right.
\end{equation}
\end{lemma}

 \begin{proof}[Proof of Lemma \ref{3.1}]
Since $\Omega$ is of class $C^{1,1}$, we obtain by paraphrasing the proof of \cite[formula $(2.12)$, page 137]{Fr}:
 \[
  |\partial _\nu E(x,t;\xi ,\tau )|\leq C(t-\tau )^{-\mu}|x-\xi |^{-n+2\mu},
 \]
 for any $\mu >0$, and then
 \begin{equation}\label{3.10}
 | M_1(x,t;\xi ,\tau )|\leq C(t-\tau )^{-\mu}|x-\xi |^{-n+2\mu}.
 \end{equation}
 We assume first that $1/2<\mu <1$. Since
 \[
 |M_2(x,t;\xi ,\tau )|\leq \int_\tau ^t\int_{\partial \Omega}|M_1(x,t;\eta ,\sigma )||M_1(\eta ,\sigma ;\xi ,\tau )|d\eta d\sigma ,
 \]
 Hence, \eqref{3.10} leads
 \begin{equation}\label{3.11}
 |M_2(x,t;\xi ,\tau )|\leq C^2\int_\tau ^t(t-\sigma )^{-\mu}(\sigma -\tau )^{-\mu}d\sigma \int_{\partial \Omega}|x-\eta|^{-n+2\mu}|\xi-\eta |^{-n+2\mu}d\eta .
 \end{equation}
By Lemma \ref{lemma3.2},
 \begin{equation}\label{3.12}
 \int_{\partial \Omega}|x-\eta|^{-n+2\mu}|\xi-\eta |^{-n+2\mu}d\eta \leq \left\{ \begin{array}{lll} \widehat{C}|x-\xi |^{-n+4\mu -1}\;\; &\textrm{if}\;\; n\geq 3\;\; \textrm{or}\;\; n=2 \;\; \textrm{and}\;\; \frac{1}{2}<\mu <\frac{3}{4},\\ \widehat{C} &\textrm{if}\;\; n =2\;\; \textrm{and}\;\; \frac{3}{4}<\mu <1.\end{array}\right.
 \end{equation}
 On the other hand
 \begin{equation}\label{3.13}
 \int_\tau ^t(t-\sigma )^{-\mu}(\sigma -\tau )^{-\mu}d\sigma =(t-\tau)^{-\mu +(1-\mu )}\int_0^1s^{-\mu}(1-s)^{1-\mu }ds=(t-\tau)^{-\mu +(1-\mu )}B(1-\mu ,1-\mu ).
 \end{equation}

We plug \eqref{3.12} and \eqref{3.13} into \eqref{3.11}, and we obtain
 \[
 |M_2(x,t;\xi ,\tau )|\leq C^2\widehat{C}(t-\tau)^{-\mu +(1-\mu )}B(1-\mu ,1-\mu )|x-\xi |^{-n+2\mu +(2\mu -1)},\;\; \textrm{if}\; n\geq 3\; \textrm{or}\; n=2 \; \textrm{and}\; \frac{1}{2}<\mu <\frac{3}{4}
 \]
 and
  \[
 |M_2(x,t;\xi ,\tau )|\leq C^2\widehat{C}(t-\tau)^{-\mu +(1-\mu )}B(1-\mu ,1-\mu ),\;\; \textrm{if}\; n =2\; \textrm{and}\; \frac{3}{4}<\mu <1.
 \]

Let $\ell (n)$ be the smallest integer $\ell$ so that $n+1<2\ell$ and fix $ \frac{n+1}{2\ell(n)}<\mu <1$. Then an induction argument yields 
 \[
 |M_\ell (x,t;\xi ,\tau )|\leq C^\ell \widehat{C}^{\ell -1}(t-\tau)^{-\mu +(\ell -1)(1-\mu )}\frac{\Gamma (1-\mu )^\ell}{\Gamma (\ell (1-\mu ))},\;\; \ell \geq \ell(n).
 \]

\smallskip
By Stirling's  formula 
 \[ 
 \Gamma (\ell (1-\mu))\sim (e^{-1}(\ell (1-\mu )-1))^{\ell (1-\mu )-1}\sqrt{2\pi (\ell (1-\mu )-1))}\;\; \mbox{as}\; \ell \rightarrow +\infty ,
 \]
implying  that the series
 \[
S= \sum_{\ell \geq \ell (n)}C^\ell \widehat{C}^{\ell -1}T^{-\mu +(\ell -1)(1-\mu )}\frac{\Gamma (1-\mu )^\ell}{\Gamma (\ell (1-\mu ))}
\]
converges. 

\smallskip
Clearly,
\[
|\mathcal{N}(x,t;\xi ,\tau )| \leq\ \sum_{\ell =1}^{\ell (n)-1}|M_\ell(x,t;\xi ,\tau )|+S.
\]
Therefore, it is enough to prove that \[ \widetilde{\mathcal{N}}(x,t;\xi ,\tau )= \sum_{\ell =1}^{\ell (n)-1}|M_\ell(x,t;\xi ,\tau )|\] satisfies \eqref{3.8}.
 To this end, we observe that
 \begin{align*}
 |M_2(x,t;\xi ,\tau )| \leq \int_\tau ^{(\tau +t)/2}\int_{\partial \Omega}|M_1(x,t;\eta ,\sigma )|&|M_1(\eta ,\sigma ;\xi ,\tau )|d\eta d\sigma 
 \\
 &+ \int_{(\tau +t)/2} ^t\int_{\partial \Omega}|M_1(x,t;\eta ,\sigma )||M_1(\eta ,\sigma ;\xi ,\tau )|d\eta d\sigma .
 \end{align*}

Assume that $1/2<\mu < n/2$ and  pick $1/2<\alpha <\min (1,(n/2-\mu)+1/2)$. 
From Lemma \ref{lemma3.2}, we have
 \begin{align*}
 \int_\tau ^{(\tau +t)/2}\int_{\partial \Omega}|M_1(x,t;\eta ,\sigma )||M_1(\eta ,\sigma ;\xi ,\tau )|d\eta d\sigma &\leq C^2\int_\tau ^{(\tau +t)/2}(\sigma-\tau )^{-\alpha}(t-\sigma )^{-\mu}d\sigma 
 \\
 &\hskip 2cm \int_{\partial \Omega}|x-\eta|^{-n+2\mu}|\xi-\eta |^{-n+2\alpha }d\eta 
 \\
 &\leq C^2\left(\frac{t-\tau }{2}\right)^{-\mu}\int_\tau ^{(\tau +t)/2}(\sigma-\tau )^{-\alpha}d\sigma |x-\xi |^{-n+2\mu+2\alpha -1}
 \\
 &\leq C^2\left(\frac{t-\tau }{2}\right)^{-\mu -\alpha +1} |x-\xi |^{-n+2\mu+2\alpha -1}
 \\
 &\leq C'(t-\tau )^{-\mu} |x-\xi |^{-n+2\mu}.
 \end{align*}
 Similarly,
 \[
 \int_{(\tau +t)/2} ^t\int_{\partial \Omega}|M_1(x,t;\eta ,\sigma )||M_1(\eta ,\sigma ;\xi ,\tau )|d\eta d\sigma \leq C(t-\tau )^{-\mu} |x-\xi |^{-n+2\mu}.
 \]
 Thus, $M_2$ satisfies \eqref{3.8}. We repeat the previous  argument to deduce that also $\widetilde{\mathcal{N}}$ obeys \eqref{3.8}.
  \end{proof}
  
 \begin{proof}[Proof of \eqref{3.4}]
 Let
 \begin{align*}
&\mathcal{N}_k(x,t; \xi ,\tau )=-2\partial _\nu E (x,t;\xi ,\tau )-2\sum_{\ell \geq 1}^k\int_\tau ^t\int_{\partial \Omega}M_\ell (x,t;\eta ,\sigma )\partial _\nu E (\eta ,\sigma ;\xi ,\tau )d\eta d\sigma ,
\\
&\varphi _k (x,t)=-2F_\tau (x,t)-2\sum_{\ell \geq 1}^k\int_\tau ^t\int_{\partial \Omega} M_\ell (x,t;\xi ,\sigma )F_\tau (\xi ,\sigma )d\xi d\sigma  .
\end{align*}
In light of Lemma \ref{lemma3.2} and with the help of Lebesgue's dominated convergence theorem, we can assert that
\[
\int_\Omega \mathcal{N}_k(x,t; \xi ,\tau )\psi (\xi )d\xi \longrightarrow \int_\Omega \mathcal{N}(x,t; \xi ,\tau )\psi (\xi )d\xi \;\; \mbox{as}\; k\longrightarrow +\infty .
\]
According to Funini's theorem
\[
\varphi _k (x,t)= \int_\Omega \mathcal{N}_k(x,t; \xi ,\tau )\psi (\xi )d\xi .
\]
But  $\varphi _k (x,t)\rightarrow \varphi  (x,t)$ when $k$ tends to infinity. Then the uniqueness of the limit yields
\[
\varphi (x,t)=\int_\Omega \mathcal{N}(x,t; \xi ,\tau )\psi (\xi )d\xi .
\]
 \end{proof}

We are now ready to prove 

 \begin{theorem}\label{theorem3.1}
Under the assumption that $\Omega$ obeys the chain condition, the Neumann Green function $G$ satisfies the Gaussian lower bound:
 \begin{equation}\label{3.14}
 \mathscr{G}_C(x-\xi ,t-\tau )\leq G(x,t;\xi ,\tau ),\;\; (x,t;\xi ,\tau )\in Q^2\cap\{t>\tau\}.
 \end{equation}
 \end{theorem}
 
 \begin{proof}
 Let 
 \[
 G_0(x,t;\xi ,\tau )=\int_\tau ^t\int_{\partial \Omega}E (x,t;\eta ,\sigma )\mathcal{N}(\eta ,\sigma ;\xi ,\tau )d \eta d\sigma.
 \]
 From the Gaussian upper bound for $E$ we obtain in a straightforward way that, for any $\beta >0$,
 \[
 |E(x,t;\xi ,\tau )|\leq C(t-\tau )^{-\beta }|x-\xi |^{-n+2\beta}.
 \]
 On the other hand, by Lemma \ref{lemma3.1},
 \[
| \mathcal{N}(\eta ,\sigma ;\xi ,\tau )|\leq C(t-\tau )^{-\mu }|x-\xi |^{-n+2\mu}.
 \]
 where $\frac{1}{2}<\mu < \frac{n}{2}$.
 
 \smallskip
 We fix $0<\epsilon <\frac{1}{2}$ and $0<\alpha <\frac{1}{2}$. In the preceding inequalities, we take $\mu=\frac{n}{2}-\epsilon$ and $\beta =1+\epsilon -\alpha$. In light of the fact that $-n+2\mu +2\beta -1=1-2\alpha >0$, we get from Lemma \ref{lemma3.2}
 \[
 |G_0(x,t;\xi ,\tau )|\leq C(t-\tau )^{-n/2 +\alpha } .
 \]
But, we know from \eqref{2.11} that
\[
E (x,t;\xi ,\tau )\geq C(t-\tau )^{-n/2},\;\; (x,t;\xi ,\tau )\in P^2,\; t>\tau ,\; \widehat{C}|x-\xi |^2<t-\tau .
\]
Hence,
\begin{align*}
G(x,t;\xi ,\tau )&\geq E (x,t;\xi ,\tau )- |G_0(x,t;\xi ,\tau )|
\\
&\geq C(t-\tau )^{-n/2}(1-\widetilde{C}(t-\tau )^{\alpha}),\; t>\tau ,\; \widehat{C}|x-\xi |^2<t-\tau .
\end{align*}

Consequently, we find $\delta >0$ so that
\[
G (x,t;\xi ,\tau )\geq C(t-\tau )^{-n/2},\;\; (x,t;\xi ,\tau )\in Q^2,\; 0<t- \tau \leq \delta ,\; \widetilde{C}|x-\xi |^2<t-\tau .
\]
Or equivalently
\begin{equation}\label{3.26}
G (x,t;\xi ,\tau )\geq C(t-\tau )^{-n/2},\;\; (x,t;\xi ,\tau )\in Q^2,\; 0<t- \tau \leq \delta ,\; |x-\xi |<\widehat{C}(t-\tau )^{1/2}.
\end{equation}
 As $\Omega$ has the chain condition, there exists a constant $c>0$, independent on $x$ and $\xi$, such that for any positive integer $k$ there exists a sequence $(x_i)_{0\leq i\leq k}$ of points in $\Omega$ so that $x_0=x$, $x_k=\xi$ and 
\begin{equation}\label{3.27}
|x_{i+1}-x_i|\leq \frac{c}{k}|x-\xi |,\;\; 0\leq i\leq k-1.
\end{equation}

When $2c|x-\xi |\leq \widehat{C}(t-\tau )^{1/2}$ (implying $|x-\xi |\leq \widehat{C}(t-\tau )^{1/2}$), \eqref{3.14} follows immediately from \eqref{3.26}. Therefore we may assume  that $2c|x-\xi | >\widehat{C}(t-\tau )^{1/2}$. We choose $m\geq 2$ to be the smallest integer satisfying
\[
2c\frac{|x-y|}{m^{1/2}}\leq \widehat{C}(t-\tau )^{1/2}.
\]
Let $(x_i)_{0\leq i\leq m}$ be the sequence given by \eqref{3.27} when $k=m$ and
\[
r=\frac{1}{4}\widehat{C}\left( \frac{t-\tau}{m}\right)^{1/2}.
\]
In light of the reproducing property and the positivity of $G$, we obtain
\begin{align}\label{3.28}
G(x,t;\xi ,\tau )&=\int_\Omega \ldots \int_\Omega G\left(x,t;\xi _1, \frac{(m-1)t+\tau}{m}\right)\ldots G\left(\xi _{m-1}, \frac{t+(m-1)\tau }{m} ;\xi ,\tau\right)d\xi _1\ldots d\xi_{m-1}\nonumber
\\
&\geq\int_{B(x_1,r)\cap \Omega} \ldots \int_{B(x_{m-1},r)\cap\Omega} G\left(x,t;\xi _1, \frac{(m-1)t+\tau}{m}\right)\ldots \nonumber
\\
&\hskip 7cm \ldots G\left(\xi _{m-1}, \frac{t+(m-1)\tau}{m} ;\xi ,\tau \right)d\xi _1\ldots d\xi_{m-1}.
\end{align}

Using that $\Omega$ is $C^{1,1}$-smooth, we obtain from the result in Appendix \ref{appendixA}: there exist two positive constants $d$ and $r_0$ such that, for any $z\in \overline{\Omega}$ and $0<\rho\leq r_0$,
\begin{equation}\label{3.31}
d\rho^n\leq |B(z,\rho )\cap \Omega |.
\end{equation}
We mention that Choi and Kim \cite{CK}  observed that this condition is necessary for domains having a De Giorgi-Nash-Moser type estimate at the boundary.

\smallskip
In the sequel, replacing $\widehat{C}$ by a smaller constant, we may assume that $r\leq r_0$.

\smallskip
Let $\xi _0=x$, $\xi_i \in B(x_i,r)$ and $\xi _m=\xi$. Then we have
\[
|\xi _{i+1}-\xi _i|\leq |x_{i+1}-x_i| +2r\leq c\frac{|x-\xi |}{m}+2r\leq c\frac{|x-\xi |}{m^{1/2}}+2r\leq 4r,\;\; 0\leq i\leq m-1 .
\]
Whence,
\[
|\xi _{i+1}-\xi _i|\leq \widehat{C}\left( \frac{t-\tau}{m}\right)^{1/2},\;\; 0\leq i\leq m-1 .
\]
It follows from \eqref{3.26} and \eqref{3.31} that
\begin{align*}
G(x,t;\xi ,\tau )&\geq \int_{B(x_1,r)\cap \Omega} \ldots \int_{B(x_{m-1},r)\cap\Omega}C^m\left( \frac{t-\tau}{m}\right)^{-nm/2}d\xi _1\ldots d\xi_{m-1}
\\
&\geq (dr^n)^{m-1}C^m\left( \frac{t-\tau}{m}\right)^{-nm/2}
\\
&\geq d^{m-1}\left[\frac{\widehat{C}^2}{16}\left( \frac{t-\tau}{m}\right)\right]^{n(m-1)/2}C^m\left( \frac{t-\tau}{m}\right)^{-nm/2}
\\
&\geq \widetilde{C}C^m(t-\tau)^{-n/2}.
\end{align*}
Hence
\begin{equation}\label{3.29}
G(x,t;\xi ,\tau )\geq \widetilde{C}e^{-Cm}(t-\tau)^{-n/2}.
\end{equation}
From the definition of $m$, we have
\begin{equation}\label{3.30}
m-1\leq \left(\frac{2c}{\widehat{C}}\right)^2\frac{|x-y|^2}{t-\tau}.
\end{equation}
Finally, a combination of \eqref{3.29} and \eqref{3.30} leads to \eqref{3.14} when $t-\tau \leq \delta$.

\smallskip
We complete the proof by showing that we can remove the assumption $t- \tau \leq \delta$ in \eqref{3.26}. Let then $0<t-\tau \leq T$, such that $t-\tau >\delta$,  and let $m\geq 2$ be the smallest integer such that $\delta ^{-1}(t-\tau )\leq m$.  We set
\begin{equation}\label{3.33}
r=r_0T^{-1/2}m^{-1/2}(t-\tau )^{1/2}
\end{equation}
and we denote by $p$ the smallest integer satisfying
\[
\frac{2cD}{rm}\leq p,\;\; \mbox{with}\; D=\mbox{diam}(\Omega ).
\]
If we choose $k=pm$ in \eqref{3.27}, we obtain
\[
|x_{i+1}-x_i|\leq \frac{c|x-\xi |}{pm}\leq \frac{cD}{pm}\leq {r/2}.
\]
Let us denote by $(3.17 ^\ast )$ the inequality \eqref{3.28} in which we take $r/2$ in place of $r$, with $r$ given as in \eqref{3.33}, and $m$ changed by $pm$.

\smallskip
Taking into account that
\[
p<1+2cDr_0^{-1}T^{1/2}\delta ^{-1/2}=p^\ast ,
\]
we get, for $\xi _i\in B(x_i,r)$, $1\leq i\leq m-1$.
\begin{equation}\label{3.34}
\frac{pm|\xi _{i+1}-\xi _i|^2}{t-\tau}\leq \frac{pmr^2}{t-\tau}<p^\ast r_0^2T^{-1}.
\end{equation}
As $(pm)^{-1}(t-\tau )\leq \delta$, \eqref{3.14} holds true. Therefore, in light of \eqref{3.34}, we obtain from $(3.17 ^\ast )$
\begin{align*}
G(x,t;\xi ,\tau )&\geq \left(d\left[2^{-1}r_0T^{-1/2}m^{-1/2}(t-\tau )^{1/2}\right]^n\right)^{pm-1}C^{pm}\left[(pm)^{-1}(t-\tau )\right]^{-pnm/2}
\\
&\geq \left(d\left[2^{-1}r_0T^{-1/2}\right]^n\right)^{pm-1}C^{pm}m^{n/2}p^{pnm/2}(t-\tau )^{-n/2}
\\
&\geq \left(d\left[2^{-1}r_0T^{-1/2}\right]^n\right)^{pm-1}(t-\tau )^{-n/2}
\\
&\geq \widetilde{C}\widehat{C}^{pm}(t-\tau )^{-n/2}
\\
&\geq \widetilde{C}e^{-{pm}|\ln \widehat{C}|}(t-\tau )^{-n/2}
\\
&\geq \widetilde{C}e^{-p^\ast m^\ast |\ln \widehat{C}|}(t-\tau )^{-n/2},\;\; \mbox{with}\; m^\ast=\delta ^{-1}T+1.
\end{align*}
This estimate completes the proof.
\end{proof}

Theorem \ref{theorem3.1} can be easily extended to a Robin Green function. Indeed, if we replace the Neumann boundary condition by the following Robin boundary condition:
\[
\partial _\nu u+q(x,t)u=0\;\; \mbox{in}\; \Sigma _\tau ,
\]
where $q\in C(\Sigma _\tau )$, then $\mathcal{N}$ has to be changed by
\begin{align*}
\mathcal{N}_q(x,t; \xi ,\tau )=-2[\partial _\nu &E (x,t;\xi ,\tau )+ q(x,t)E (x,t;\xi ,\tau )]
\\
&-2\sum_{\ell \geq 1}\int_\tau ^t\int_{\partial \Omega}M_\ell (x,t;\eta ,\sigma )[\partial _\nu E (\eta ,\sigma ;\xi ,\tau )+q(\eta ,\sigma )E (\eta ,\sigma ;\xi ,\tau )]d\eta d\sigma  .
\end{align*}
Here
\begin{align*}
&M_1(x,t;\xi ,\tau )=-2[\partial _\nu E (x,t;\xi ,\tau )+ q(x,t)E (x,t;\xi ,\tau )]
\\
&M_{\ell +1}(x,t;\xi ,\tau )=\int_\tau ^t\int_{\partial \Omega}M_1(x,t;\eta ,\sigma )M_\ell (\eta ,\sigma ;\xi ,\tau )d\eta d\sigma ,\;\; \ell \geq 1.
\end{align*}
Apart the positivity of the Green function, which can be obtained by an adaptation of \cite[ Proposition 3.2]{Chou}, one can see without any difficulty, that the rest of our analysis holds true when $\mathcal{N}$ is replaced by $\mathcal{N}_q$. 

\smallskip
We already mentioned that, for parabolic operators with time-independent coefficients, a Neumann Green function is nothing else but a Neumann heat kernel.  Let us then consider  a parabolic operator of the form
\begin{equation}\label{3.24}
\mathcal{L}=\partial _j (a_{ij}(x)\partial _i\, \cdot\, )+b_k(x)\partial _k+c(x)-\partial _t,
\end{equation}
so that the following assumptions are satisfied:
\begin{align*}
&(i')\; \mbox{the matrix}\; (a_{ij}(x))\; \mbox{is symmetric for any}\; x\in \overline{\Omega}, \hskip 9cm
\\
&(ii')\; a_{ij}\in W^{1,\infty}(\Omega ),\;\; \partial _ka_{ik},\; b_k,\; c \in  C^1 (\overline{\Omega}),
\\
&(iii')\; a_{ij}(x)\xi _i\xi _j\geq \lambda |\xi |^2,\;\; (x,t)\in \overline{\Omega}  ,\; \xi \in \mathbb{R}^n,
\\
&(iv')\; \|a_{ij}\|_{W^{1,\infty}(\Omega )}+\| \partial _ka_{ik}+b_k\|_{L^\infty (\Omega )}+\|c\|_{L^\infty (\Omega )}\leq A,
\end{align*}
where $\lambda >0$ and $A>0$ are two given constants.

\smallskip
Here again, the assumptions on the coefficients of $\mathcal{L}$ are not necessarily the best possible.

\smallskip
Let $\mathfrak{a}$ be the unbounded bilinear form defined by $D(\mathfrak{a})=H^1(\Omega )$ and
\[
\mathfrak{a}(u,v)=\int_\Omega a_{ij}\partial _iu\partial _jvdx+\int_\Omega b_k\partial _kuvdx+\int_\Omega cuvdx,\;\; u,v\in D(\mathfrak{a}).
\]

We associate to $\mathfrak{a}$ the unbounded operator $\mathcal{A}$ given by 
\[
D(\mathcal{A})=\{ u\in L^2(\Omega );\; \exists v\in L^2(\Omega ): \mathfrak{a}(u,\varphi )=(v,\varphi )_2,\; \varphi \in H^1(\Omega )\},\; \mathcal{A}u:=v.
\]
Here $(\cdot ,\cdot )_2$ denotes the usual scalar product of $L^2(\Omega )$.

\smallskip
We have
\[
\int_\Omega b_k\partial _ku u \leq \frac{2}{\lambda }\int_\Omega |\nabla u |^2+\left(\frac{8\sup_k\|b_k\|_{L^\infty (\Omega)}}{\lambda}\right) \int_\Omega u^2,\;\; u\in H^1(\Omega ).
\]
This and $(iii')$ entail that $\mathfrak{a}+\kappa$ is accretive for a sufficiently large $\kappa >0$. Since $\mathfrak{a}$ is clearly densely defined and continuous on $H^1(\Omega )$, we derive from \cite[Theorem 1.52, page 29]{Ou} that $-\mathcal{A}$ is the generator of  an holomorphic semigroup $e^{-t\mathcal{A}}$.

\smallskip
Let $Q=\Omega \times (0,T)$, $\Sigma =\partial \Omega \times (0,T)$, $\psi \in C_0^\infty (\Omega )$ and $u\in C^{1,0}(\overline{Q})\cap C^{2,1}(Q)$ (\cite[Theorem2, page 144]{Fr}) be the unique solution of the IBVP
\begin{equation}\label{e1}
\left\{
\begin{array}{lll}
\mathcal{L}u =0\quad &\mbox{in}\; Q,
\\
u(\cdot\, ,0 )=\psi &\mbox{in}\; \Omega ,
\\
\partial _\nu u=0 &\mbox{on}\; \Sigma .
\end{array}
\right.
\end{equation}

By \cite[Theorem 10.9, page 341]{Br}, $u\in L^2(0,T;H^1(\Omega )\cap C([0,T];L^2(\Omega ))$, $u'\in L^2(0,T; [H^1(\Omega )]')$ and it is the unique solution of
\begin{equation}\label{e2}
\left\{
\begin{array}{ll}
\langle u'(t),v\rangle +\mathfrak{a}(u(t),v)=0\;\; \textrm{a.e.}\; t\in [0,T],\; v\in H^1(\Omega ),
\\
u(0)=\varphi
\end{array}
\right.
\end{equation}
where $\langle \cdot ,\cdot \rangle$ is the duality pairing between $H^1(\Omega )$ and its dual space $ [H^1(\Omega )]'$.

\smallskip
We set $\widetilde{u}(t)=e^{-t\mathcal{A}}\psi$, $t\geq 0$. By using $\widetilde{u}'(t)=\mathcal{A}u(t)$, $t\in [0,T]$, we obtain in a straightforward manner 
\[
\langle \widetilde{u}'(t),v\rangle +\mathfrak{a}(\widetilde{u}(t),v)=0,\;\;  t\in [0,T],\; v\in H^1(\Omega ),
\]
Using that $u(0)=\widetilde{u}(0)$, we get from the uniqueness of the solution of problem \eqref{e2} that $u=\widetilde{u}$. Hence,
\[
e^{-t\mathcal{A}}\psi (x)=\int_\Omega G(x,t;\xi ,0)\psi (\xi )d\xi ,\; 0<t \leq T.
\]
We rewrite this equality as follows:
\[
e^{-t\mathcal{A}}\psi (x)=\int_\Omega K(x,\xi ,t)\psi (\xi )d\xi ,\; 0<t \leq T.
\]
The function
\[
K(x,\xi ,t)= G(x,t;\xi ,0)
\]
is usually called the heat kernel of the semigroup $e^{-t\mathcal{A}}$.

\smallskip
We have as an immediate consequence of  Theorem \ref{theorem3.1}:
 \begin{corollary}\label{corollary3.1}
When $\Omega$ possesses the chain condition, the Neumann heat kernel $K$ satisfies the Gaussian lower bound:
 \begin{equation}\label{3.25}
 \mathscr{G}_C(x-\xi ,t)\leq K(x,\xi ,t),\;\; (x,\xi )\in \Omega^2 ,\; 0<t\leq T.
 \end{equation}
 \end{corollary}
 
A Gaussian lower bound for the Neumann heat kernel was proved in \cite{COY} when $L$ is the Laplace operator. The key point is the H\"older continuity of $x\longrightarrow K(x,\xi ,t)$ which relies on the fact that $\mu-\mathcal{A}$ is an isomorphism from $H^s(\Omega )$ into $H^{s-2}(\Omega )$, for large $\mu$ and  $s>n/2+1$. We note that a quick examination of the proof in \cite{COY} shows that this result can be extended to a divergence form operator with $C^\infty$-smooth coefficients.

 \medskip
 We end this section by showing that we can obtain a strong maximum from Theorem \ref{theorem3.1}. Let $\psi \in C(\overline{\Omega})$, $f\in C(\overline{Q}_\tau)$, $g\in C(\overline{\Sigma} _\tau )$ and consider the IBVP
 \begin{equation}\label{3.32}
\left\{
\begin{array}{lll}
Lu =f\quad &\mbox{in}\; Q_\tau  ,
\\
u(\cdot\, ,\tau )=\psi &\mbox{in}\; \Omega ,
\\
\partial _\nu u=g &\mbox{on}\; \Sigma_\tau.
\end{array}
\right.
\end{equation}

\begin{corollary}\label{corollary3.2}
We assume that $\psi \geq 0$, $f\geq0$, $g\geq 0$  and at least one of the functions $\psi$, $f$ and $g$ is non identically equal to zero. If $u\in C^{0,1}(\overline{Q}_\tau)\cap C^{2,1}(Q_\tau )$ is the solution of the IBVP \eqref{3.32}, then
\[
u>0\; \mbox{in}\; \Omega \times ]\tau , t_1].
\]
\end{corollary}

\begin{proof}
Follows from Theorem \ref{theorem3.1} since (see for instance \cite[formula (3.5), page 144]{Fr})
\[
u(x,t)=\int_\Omega G(x,t;\xi ,\tau )\psi (\xi )d\xi +\int_\tau^t\int_\Omega G(x,t;\xi ,s)f(\xi ,s)d\xi ds+\int_\tau^t\int_{\partial \Omega} G(x,t;\xi ,s)g(\xi ,s)dS_\xi ds.
\]
\end{proof}

\appendix
\section{}\label{appendixA}

In this appendix we prove \eqref{3.31}. Henceforth, $\Omega$ is a bounded domain of $\mathbb{R}^n$ with boundary $\Gamma$.

\smallskip
Following \cite[Definition 2.4.1, page 50]{HP}, we introduce the notation
\[
\mathscr{C}(y,\xi ,\epsilon )=\{ z\in \mathbb{R}^n;\; (z-y)\cdot \xi \geq (\cos \epsilon )|z-y|,\; 0<|y-z|<\epsilon \},
\]
where $y\in\mathbb{R}^n$, $\xi \in \mathbb{S}^{n-1}$ and $0<\epsilon $. That is, $C(y,\xi ,\epsilon)$  is the cone, of dimension $\epsilon$,  with vertex $y$, aperture $\epsilon$ and directed by $\xi$.

\smallskip
We say that $\Omega$ has the $\epsilon$-cone property if
\begin{equation}\label{a1}
\textrm{for any}\; x\in \Gamma ,\; \textrm{there exists}\; \xi _x\in \mathbb{S}^{n-1}\; \textrm{so that, for all}\; y\in \overline{\Omega}\cap B(x,\epsilon),\; \mathscr{C}(y,\xi _x,\epsilon )\subset \Omega .
\end{equation}

Assume that $\Omega$ has the $\epsilon$-cone property, for some $0<\epsilon $. By using the compactness of $\Gamma$, we find a finite number of points of $\Gamma$, $x_1,\ldots x_p$, so that $\Gamma =\bigcup_k\left[\Gamma \cap B(x_k,\epsilon /2)\right]$ and \eqref{a1} is satisfied for each $x_i$, $i=1,\ldots ,p$. 
Let $K=\overline{\Omega}\setminus \bigcup_k B(x_k,\epsilon /2)$. Then, $ 0<\varrho=\textrm{dist}(K,\Gamma ) (<\epsilon )$ and 
therefore,  for each  $x\in K$, we have $B(x,\varrho )\subset \Omega$. We deduce from this observation that, for each $x\in \overline{\Omega}$ and $0<r<\varrho$, $\Omega \cap B(x,r)$ contains a cone of dimension $r$ and aperture $\epsilon$. It is then straightforward to get the following inequality:
\[
|\Omega \cap B(x,r)|\geq cr^n,\;\; 0<r<\varrho ,
\]
for some constant $c=c(n,\varrho )$.

\smallskip
We complete the proof of \eqref{3.31} by using the following theorem.

\begin{theorem}\label{theorem-a1}
$\Omega$ has the  $\epsilon$-cone property, for some $0<\epsilon $, if and only if its boundary $\Gamma$ is Lipschitz.
\end{theorem}
We refer to \cite[Theorem 2.4.7, page 53]{HP} for a detailed proof of this theorem.


\vskip .5cm

\end{document}